\documentclass{article}
\usepackage{graphicx}
\usepackage{subcaption}
\usepackage{tikz}
\usepackage{enumerate}
\usepackage{algorithmic}
\usepackage{algorithm}
\usepackage{comment}
\usepackage{booktabs}
\usepackage{comment}
\usepackage{amsthm}
\usepackage{amsmath}
\usepackage{amssymb}
\usepackage[hyperfootnotes=false]{hyperref}
\usepackage{arxiv}
\usepackage{fullpage}
\newtheorem{theorem}{Theorem}

\newtheorem{proposition}[theorem]{Proposition}
\newtheorem{corollary}[theorem]{Corollary}

\usepackage{libertine}
\usepackage[scaled=0.9]{FiraMono}

\theoremstyle{remark}

\usepackage{hyperref}

\hypersetup{
    colorlinks=true,
    linkcolor=blue,
    filecolor=magenta,      
    urlcolor=cyan,
    citecolor=blue
}

\author{
\href{https://orcid.org/0009-0009-7031-9993}
{\includegraphics[scale=0.06]{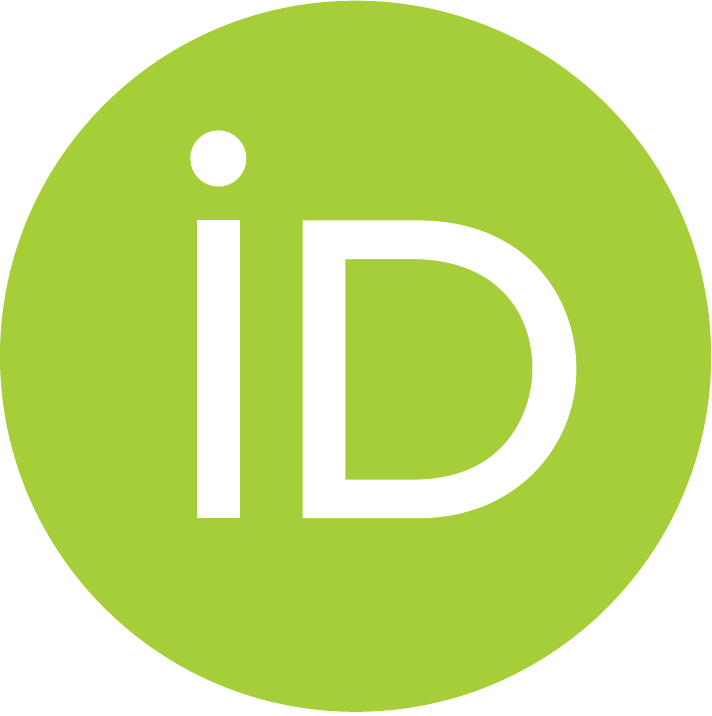}}\hspace{1mm}Gennesaret Tjusila\\
\texttt{tjusila@campus.tu-berlin.de} \\
\And
\href{https://orcid.org/0000-0002-6284-3033}{\includegraphics[scale=0.06]{orcid_id_icon.pdf}}\hspace{1mm}Mathieu Besançon\footnotemark[2]\\
\texttt{besancon@zib.de} \\
\And
\href{https://orcid.org/0000-0001-7270-1496}{\includegraphics[scale=0.06]{orcid_id_icon}}\hspace{1mm}Mark Turner\footnotemark[1]\hspace{2mm}\thanks{Zuse Institute Berlin, Department of Mathematical Optimization, Takustr. 7, 14195 Berlin} \\
\texttt{turner@zib.de} \\
\And
\href{https://orcid.org/0000-0002-1967-0077}{\includegraphics[scale=0.06]{orcid_id_icon.pdf}}\hspace{1mm}Thorsten Koch\footnotemark[1]\hspace{2mm}\footnotemark[2] \\
\texttt{koch@zib.de}
}
\title{How Many Clues To Give? A Bilevel Formulation For The Minimum Sudoku Clue Problem}

\begin{document}
\maketitle
\begin{abstract}
It has been shown that any 9 by 9 Sudoku puzzle must contain at least 17 clues to have a unique solution. This paper investigates the more specific question: given a particular completed Sudoku grid, what is the minimum number of clues in any puzzle whose unique solution is the given grid? We call this problem the Minimum Sudoku Clue Problem (MSCP).
We formulate MSCP as a binary bilevel linear program, present a class of globally valid inequalities, and provide a computational study on 50 MSCP instances of 9 by 9 Sudoku grids. Using a general bilevel solver, we solve 95\% of instances to optimality, and show that the solution process benefits from the addition of a moderate amount of inequalities. Finally, we extend the proposed model to other combinatorial problems in which uniqueness of the solution is of interest.
\end{abstract}

\section{Introduction}
The Sudoku puzzle first appeared in the May 1979 edition of \emph{Dell Pencil Puzzle and Word Games} \cite{jD06}. Given a square integer $n$, the puzzle is given on a $n \times n$ grid divided into $n$ subgrids each of size $\sqrt{n} \times \sqrt{n}$.
As input, some cells are already filled with numbers between 1--n. The goal of the puzzle is to fill the rest of the cells such that each number between 1--n appears exactly once in each row, column, and subgrid. An example of a Sudoku puzzle along with its solution is given in Figure~\ref{F:SudokuSample}. For most Sudoku puzzles, uniqueness of the solution is a desirable property. We call such puzzles \emph{valid}. It is fairly easy to construct examples of $9 \times 9$ Sudoku puzzles with $77$ clues and multiple solutions (such as removing the entries marked in green in Figure~\ref{F:UnavoidableSets}). One can also observe that any puzzle with at least $78$ clues will always have a unique solution.

A natural question that arises is: what is the minimum number of clues that a valid puzzle can have? It is shown in \cite{gM14} that the answer to this question is $17$ clues. But what if the puzzle designer already has a solution grid in mind? This motivates the Minimum Sudoku Clue Problem (MSCP): what is the minimum number of clues on any valid puzzle for a given Sudoku grid?

In this paper, we make four key contributions. First, we formulate the MSCP as a binary bilevel linear program, allowing the use of generic integer bilevel methods and solvers, which to the best of our knowledge is a first in the literature. Second, we present unavoidable set inequalities, a set of globally valid inequalities, which we add at the start of the solving process to improve solver performance. Third, we provide computational results over a set of Sudoku grids to show the viability of our approach. Finally, we generalize our model to other problems which fulfill some assumption in the Fewest Clue Problem (FCP) class introduced in \cite{eD18}.
We note that this paper is an extension of the first author's thesis work \cite{gT22}.

\begin{figure}[h]
    \begin{center}
        \begin{subfigure}[b]{.48\linewidth}
        \begin{center}
        \begin{tikzpicture}[scale=0.5, every node/.style={scale=0.6}]
        \input{sudokusample}
        \end{tikzpicture}
        \end{center}
        \caption{A Sudoku Puzzle}\label{F:SudokuPuzzle}
        \end{subfigure}
        \begin{subfigure}[b]{.48\linewidth}
        \begin{center}
        \begin{tikzpicture}[scale=0.5, every node/.style={scale=0.6}]
        \input{unavoidablesets}
        \end{tikzpicture}
        \end{center}
        \caption{The Solution Grid and its Unavoidable Sets}\label{F:UnavoidableSets}
        \end{subfigure}
        \caption{A Sudoku puzzle along with its solution and unavoidable sets}\label{F:SudokuSample}
    \end{center}
\end{figure}

\section{Related Work} \label{S:RelatedWork}

The problem of counting the total number of $n \times n$ Sudoku grids is an open problem. For $n=9$, it was shown in \cite{bF06} that the number of Sudoku grids is around $6.671 \times 10^{21}$. A natural upper bound arises by considering that Sudoku grids are a subset of Latin squares with additional subgrid constraints. The enumeration of Latin squares has been extensively studied in the literature \cite{bM05}, which has pushed similar studies for Sudoku \cite{dB18, aH07}. Many of these grids are equivalent under transformations such as relabeling of digits and rotations. We call the lexicographically smallest Sudoku grid that is equivalent to a given grid under these transformations the \emph{minlex form} of the Sudoku grid \cite{cL12}.
Taking these transformations into account, the number of $9 \times 9$ Sudoku grids is reduced to around $5.47 \times 10^{9}$ \emph{essentially different} grids \cite{eR06}.
While our work focuses on the minimum number of clues for a \emph{given} Sudoku grid, the minimum number of clues for \emph{any} Sudoku grid has been shown to be $17$ through a computer-assisted proof~\cite{gM14}.
A list of nearly $50000$ Sudoku puzzles with $17$ clues is collected by Gordon Royle \cite{gR05}. This collection is only a fraction of the possible number of Sudoku grids, heavily suggesting that most Sudoku grids do not have a $17$ clue valid puzzle.
Minimum bounds for the $4\times 4$ number of clues have been derived through an algebraic process in \cite{aF13} by encoding the combinatorial problem as a polynomial and analyzing its structure.
Research in this direction for $9 \times 9$ grids has focused on analyzing the underlying graph structure of the Sudoku grid \cite{jC14, gL22} and characterizing valid Sudoku puzzles using formal logic \cite{dM22}.

Finding a solution to a general $n \times n$ Sudoku puzzle is ASP-complete, which implies NP-completeness of the decision problem as well as \#P-completeness to count the solutions \cite{tY03}. However, practical methods for solving Sudoku puzzles of size $9 \times 9$ exist in the literature \cite{lC14}. There has also been recent research in making algorithms that solve Sudoku puzzles explainable for humans \cite{bJ08}.
Given a Sudoku grid, the decision problem ``is there a setting of at most $k$ clues such that the only solution is the given grid?" is a member of the class of problem ``Fewest Clue Problem" (FCP) and has been shown to be $\Sigma_2^P$-complete~\cite{eD18}. Mixed-integer bilevel linear programming has also been shown in \cite{sD12,rJ85} to be $\Sigma_2^P$-complete. Therefore, transforming MSCP into a binary bilevel linear program retains the same complexity but allows for a general solving method.

To the best of our knowledge, all existing software libraries for solving MSCP are problem-specific and created by the Sudoku community, see \cite{d23} for an example. The software uses pattern-matching algorithms to quickly find so-called \emph{unavoidable sets}, such as described in \cite{gM14}.
An unavoidable set is defined as a set of cells whose contents if removed will result in an invalid Sudoku puzzle. An example of such sets would be the cells marked in green, red, or blue in Figure~\ref{F:UnavoidableSets}.
Given a set of unavoidable sets $S$, we call a set of cells $H$ a \emph{hitting set} if for every unavoidable set in $S$ at least one cell is contained in $H$.
Once a large enough set of unavoidable sets has been generated, one can enumerate over all hitting sets of these unavoidable sets, starting from ones with minimal cardinality until a valid puzzle is found.
Although in theory, enumerating unavoidable sets is expensive, specialized algorithms are often fast in practice owing to additional problem-specific methods, e.g.~exploiting equivalence classes of Sudoku grids. We also highlight that enumeration of hitting set, in particular minimal hitting set, is an active area of research \cite{aG17}.
In contrast to existing software, our work uses a general mathematical optimization approach to solve MSC. We will use integer linear programming models to find unavoidable sets and generate valid inequalities to speed up the bilevel-solving process.

\section{Integer Bilevel Linear Formulations of Minimum Sudoku Clue Problem}\label{sec:bilevel}

We now formulate MSCP for a Sudoku grid of size $n \times n$ where $n$ is a square number. Let $x_{ijk}$ be a set of binary decision variables where $i,j,k \in [n] := \{1, \dots, n\}$. The variable $x_{ijk}$ takes value one if cell $(i,j)$ has entry $k$ and zero otherwise. The variables construct an $n \times n$ Sudoku grid if they satisfy
\begin{align*}
     & \sum_{k = 1}^n x_{ijk} = 1, & \quad \forall\ i,j \in [n]\tag{$G0$}\\
     & \sum_{j = 1}^n x_{ijk}  = 1, & \quad\forall\ i,k \in [n]\tag{$G1$}\\
     & \sum_{i = 1}^n x_{ijk} = 1, & \quad\forall\ j,k \in [n] \tag{$G2$}\\
    & \sum_{\substack{i =sp-s+1}}^{sp} \sum_{j = sq-s+1}^{sq} x_{ijk} = 1, &\quad
\forall\ p,q\in [s] \text{ and } k \in [n]\tag{$G3$}\\
\end{align*}
where $s := \sqrt{n}$.
This is the standard Sudoku integer linear program (ILP) formulation found in the literature, see \cite{tH15,tK05}.

Let $G$ be a Sudoku grid given as an $n \times n$ matrix with entries in $[n]$. The leader problem of our binary bilevel linear program will act as a ``puzzle setter'', and determine which entries of the Sudoku grid are given as clues. The follower problem will act as an ``adversary'' that tries to find a solution different from the given Sudoku grid. Concretely, our model is as follows
\begin{align*}
     \min_{x,y,z} \quad&\sum_{i=1}^n\sum_{j=1}^ny_{ij}\\
    \text{s.t.} \quad&z = 1 \tag{$V1$}\\
    &y_{ij} \in \{0,1\},\quad\forall\ i,j \in [n]\\
    &(x,z)\in S(y) 
\end{align*}
where $S(y)$ is the set of optimum solutions to the $y$-parameterized follower problem
\begingroup
\allowdisplaybreaks
\begin{align*}
     \min_{x,z}\quad&z\\
    \text{s.t.}\quad&\text{$(G0)-(G3)$}\\
    &x_{ijG_{ij}} \geq y_{ij}, \quad \forall\ i,j \in [n]\tag{$F1$}\\
    &\sum_{i = 1}^n \sum_{j = 1}^n x_{ijG_{ij}} - z\leq n^2-1 \tag{$N1$}\\
    &x_{ijk},z \in \{0,1\}, \quad\forall\ i,j,k \in [n].
\end{align*}
\endgroup

 The leader decision variable, $y_{ij}$, determines whether the entry of a cell $(i,j)$ is given to the follower problem as a clue. The objective function of the leader problem is the number of clues given. Constraint $(F1)$ requires the follower problem to adhere to these given clues. Constraint $(N1)$ requires that the Sudoku grid defined by the set of decision variables $x_{ijk}$ with $i,j,k \in [n]$ is different from $G$.
 This constraint can be relaxed by setting $z$ to one and taking a penalty. The intuition of the leader constraint $(V1)$ is as follows: The objective of the follower is to minimize this penalty. If the puzzle determined by the leader problem has multiple solutions, the follower can find a feasible solution with a penalty of zero. However, if this is not possible, then the puzzle determined by the leader is a valid puzzle and the only option the follower has is to take the penalty.
 
Finally, we highlight that the high-point relaxation is always trivially achieved by setting $z = 1$, $y_{ij}= 0$ for all $i,j \in [n]$ and $x$ to be another Sudoku grid not equal to $G$, by permuting digits for instance.
This weakness of the relaxation suggests the hardness of the bilevel problem.

\section{Strengthening The Bilevel Formulation Through Valid Inequalities}

Consider the Sudoku grid given in Figure~\ref{F:UnavoidableSets}. We can swap the $3$'s and $8$'s in the green marked cells to get a new Sudoku grid $G'$ that has the same entries except for the cells marked in green. Thus, any valid puzzle $P$ must have at least one clue in one of the green-marked cells. Similarly, we observe that it is possible to change the entries of cells marked in blue or red. Thus, there must also be at least one clue in the cells marked red and one clue in the cells marked blue. We call a set of cells $U$ an \emph{unavoidable set} for a Sudoku grid $G$ if there exists a Sudoku grid $G' \neq G$ that differs from $G$ only on cells in $U$. We call an unavoidable set \emph{minimally unavoidable} if it contains no subset that is again unavoidable. In what follows, we represent Sudoku grids as $n \times n$ matrices with entries from $[n]$ and Sudoku puzzles as $n \times n$ matrices with entries from $[n] \cup \{0\}$ where $0$ marks an empty cell.

\begin{proposition}\label{P:UnavoidableSet}
Let $G$ be a grid and $P$ a puzzle such that $G$ is a solution of $P$. Then $P$ is a valid puzzle, if and only if, for every minimally unavoidable set $U$ of $G$ there exists a cell $(i,j) \in U$ that is given as a clue, i.e., $P_{ij} \neq 0$.
\end{proposition}
\begin{proof}
 To show sufficiency suppose that there exists a minimally unavoidable set $U$ such that ${P_{ij} = 0}$ for all cells $(i,j) \in U$. By definition there exists a Sudoku grid $G' \neq G$ which differs from $G$ only in the entries of cells that are in $U$. Since $P_{ij} = 0$ for all cells $(i,j) \in U$ then $G'$ is also a solution of $P$. Thus $P$ is not a valid puzzle.
 
 To show necessity, suppose that $P$ is not a valid puzzle and there exists a Sudoku grid $G'$ which is a solution of $P$ and $G' \neq G$. We define 
 \[
 U := \{ (i, j) \in \{1, \dots, n\}^2 \mid G'_{ij} \neq G_{ij} \}
 \]
 as the set of cells whose entry in $G$ is different from its entry in $G'$. By construction, $U$ is an unavoidable set and $P_{ij} = 0$ for all $(i,j) \in U$, as otherwise, their entries would be identical. If $U$ is minimally unavoidable then we are done, otherwise, a subset of $U$ is again unavoidable. Since $U$ is finite, we can iterate the process until we end up with a minimally unavoidable set.
\end{proof}

\begin{corollary}
Let $G$ be a Sudoku grid and $U$ be an arbitrary minimal unavoidable set. Then, the inequality
\begin{equation}
    \sum_{(i,j) \in U}y_{ij} \geq 1 \tag{U}
\end{equation}
is a globally valid inequality for the leader of our bilevel program. We call this inequality the \textbf{unavoidable set inequality} corresponding to $U$
\end{corollary}

We give a method to generate the set of unavoidable sets $\mathcal{U}$. Let $m \in \mathbb{N}$ with $m\geq 1$. Consider the \text{$m$-parameterized} integer linear program,
\begin{alignat*}{3}
     \min_x &\quad0\\
    \text{s.t.}&\quad\text{$(G0)-(G3)$} \\
     &\sum_{i = 1}^n \sum_{j = 1}^n x_{ijG_{ij}} = n^2 - m \tag{$D1$}\\
     &x_{ijk} \in \{0,1\},\quad \forall\ i,j,k \in [n],
\end{alignat*}
The integer linear program gives us a Sudoku grid $G'$ which differs from $G$ in exactly $m$ entries. We get that
\[
 U := \{ (i, j) \in \{1, \dots, n\}^2 \mid G_{ij} \neq G'_{ij} \}.
\]
is an unavoidable set of G by construction. We start with $m = 1$ and repeatedly solve the ILP, adding in each iteration the no-good cut constraint
\begin{equation}
    \sum_{(i,j) \in U}x_{ijG_{ij}} \geq 1 \tag{N2}
\end{equation}
which bars the ILP from returning any $G'$, whose associated unavoidable set is a superset of $U$. The ILP will thus return a different unavoidable set of size $m$ in each iteration. Once all unavoidable sets of size $m$ have been generated, we move on to $m+1$. Note that we could have equivalently formulated this as a minimization problem.

\begin{proposition}
For the procedure described above it applies
\begin{enumerate}[(i)]
    \item At each iteration, the resulting unavoidable set will always be a minimally unavoidable set
    \item Repeating the procedure eventually yields all minimal unavoidable sets
\end{enumerate}
\end{proposition}
\begin{proof}
    To show (i), let $\bar{U}$ be an unavoidable set that is not minimal and $U \subset \bar{U}$ a minimal unavoidable set with $m := |U|$. When generating all unavoidable sets of size $m$, a no-good cut for $U$ will also be added to the formulation. Thus, any $G'$ which generates $\bar{U}$ will be infeasible because $U \subset \bar{U}$.
    We get (ii) by construction.
\end{proof}

\section{Computational Results}

In this section, we investigate the performance of our models for solving MSCP over $50$ instances of $9 \times 9$ Sudoku grids. All of our computations run on a single thread of an Intel Xeon E5-2630V4 2.2 GHz. A wall-clock time limit of 4 days and a memory limit of 16 GB was used for each run. The algorithm to generate unavoidable set cuts uses Gurobi~9.5.1 \cite{gurobi} as an ILP solver, and we use the bilevel solver from~\cite{mF17} to solve the main model, where the authors granted us a license upon request. The solver uses CPLEX~12.7 \cite{cplex} to solve linear programming relaxations. The code used for this section along with the computational results can be found in \url{https://github.com/gtjusila/minimum-sudoku}.

The 50 instances are split into two groups of 25. The first group is randomly selected from a list of Sudoku puzzles with $17$ clues \cite{gR17clue}. The second group is randomly selected from the list of Sudoku puzzles with a difficulty rating of more than $11$ (the maximum difficulty rating being $12$) maintained by the new Sudoku players forum\footnote{\url{http://forum.enjoysudoku.com/the-hardest-Sudokus-new-thread-t6539-600.html\#p277835}}. The known puzzles for all instances in this second group contain more than $20$ clues each. To get a diverse instance set, we also ensure that we select Sudokus with different minlex forms \cite{cL12}. To convert the Sudoku grids to minlex form, we use the code from \cite{mD23}.

First, we evaluate the performance of the unavoidable sets generating procedure. For each of our $9 \times 9$ instances, we generate $5000$ minimal unavoidable sets. In all instances, we generate all minimal unavoidable sets of size 16 or less. We observed no unavoidable sets of size $5$ and $7$, which leads us to conjecture that none exist for any instance. In 39 out of 50 instances, we generated all minimal unavoidable sets of size less than or equal to 17. We plot the geometric mean of the time needed to generate the kth unavoidable set over all $50$ of our instances in Figure~\ref{Fi:CutGenerationTime}. We see that generally, the time needed to generate an unavoidable set increases as $k$ gets larger. An interesting feature of the figure is the periodic peaks. Looking deeper into the result of individual instances, we see that as we try to enumerate all minimal unavoidable sets of size $n \in \mathbb{N}$, the time increases in each iteration. This is expected as in each iteration there are fewer and fewer minimal unavoidable sets of size $n$ available, and thus, they become increasingly hard to find. To visualize this effect, we computed the average number of unavoidable sets less than or equal to $n$ for $n = 11, \dots, 17$ (for the instance in which not all unavoidable sets of size $17$ have been found, we assume the number of unavoidable set of size less than $17$ to be $5000$) and drew them as vertical lines in Figure~\ref{Fi:CutGenerationTime}. One can think of these lines as the average point where an instance switches from searching for unavoidable sets of size $n$ to size $n+1$. The leftmost line represents $n = 11$.

\begin{figure}[]
    \begin{center}
        \begin{subfigure}[b]{.49\linewidth}
        \begin{center}
\includegraphics[width=\textwidth]{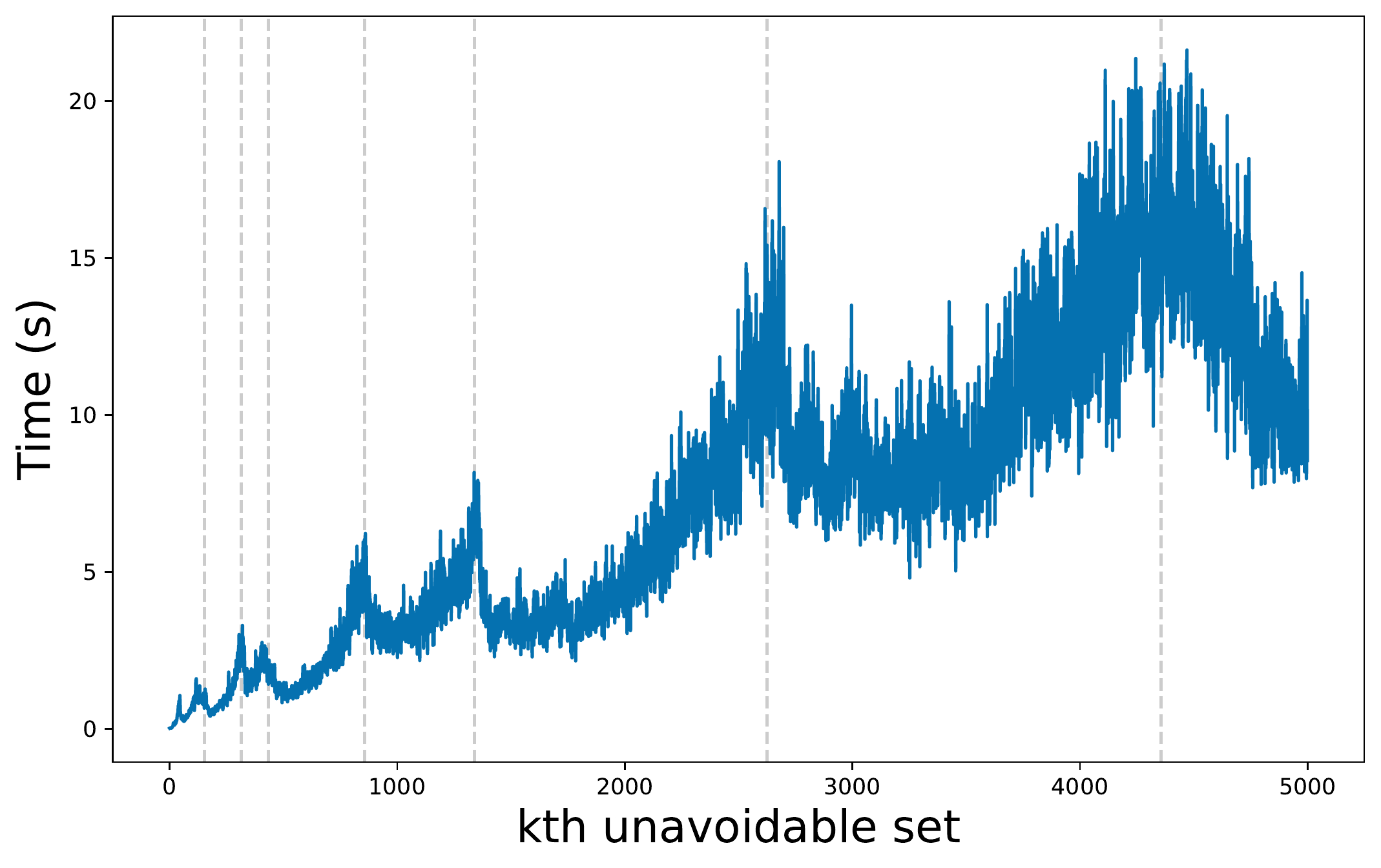}
\caption{Geometric mean of time needed to generate the kth unavoidable set}
\label{Fi:CutGenerationTime}
        \end{center}
        \end{subfigure}
        \begin{subfigure}[b]{.49\linewidth}
        \begin{center}\includegraphics[width=\textwidth]{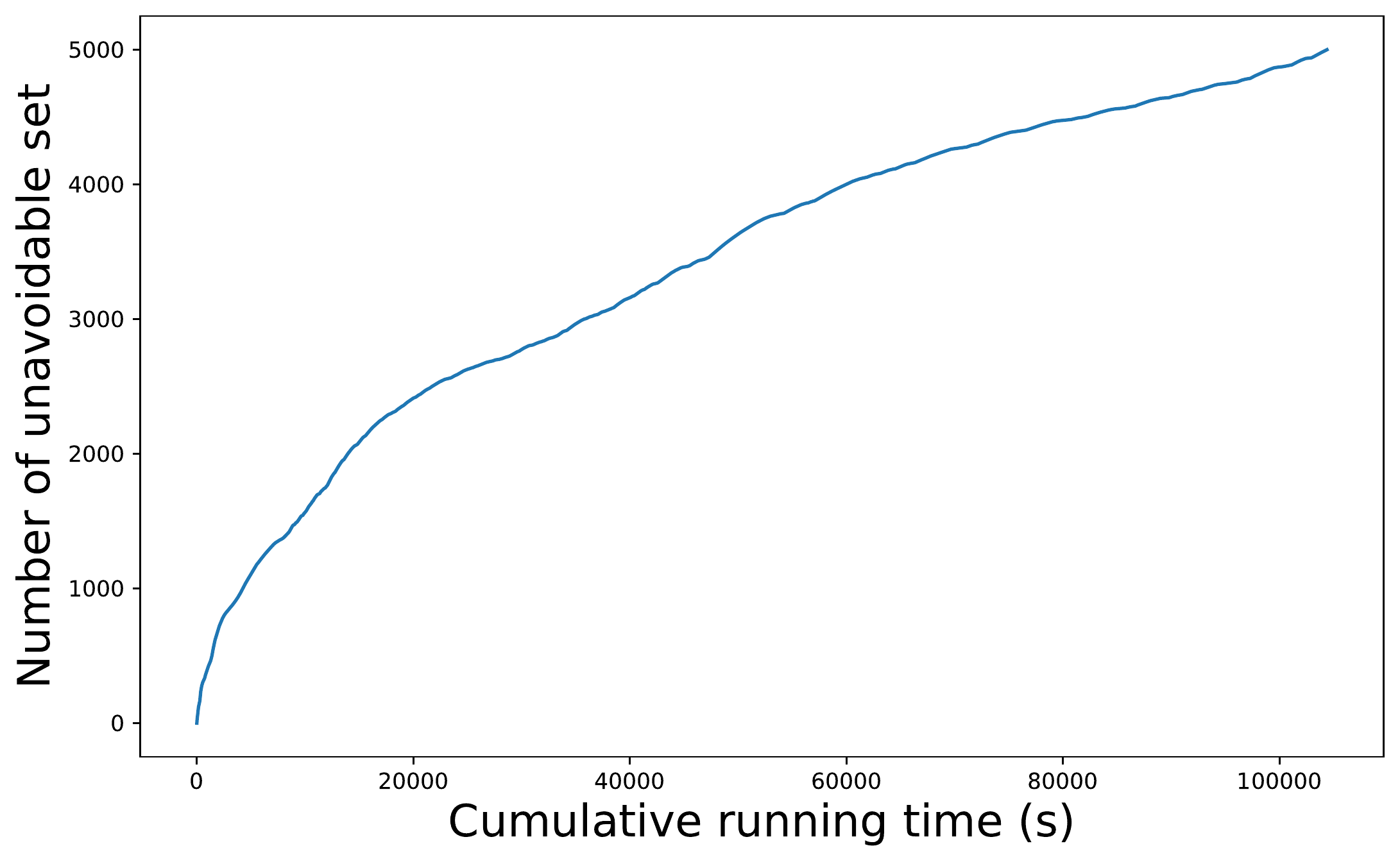}
\caption{The number of unavoidable set generated as a function of time}
\label{Fi:YieldCurve}
        \end{center}
        \end{subfigure}
    \end{center}
    \caption{Experiment Results For Cut Generation}
\end{figure}
\begin{figure}[]
\end{figure}
\begin{table}[h]
    \centering
\begin{tabular}{rr}
\toprule
 generation time [s] &  \# of unavoidable sets \\
\midrule
               $\leq1$ &       30763 \\
              $1-10$ &       136009 \\
              $10-30$ &       42389 \\
              $30-60$ &       26092 \\
             $60-300$ &        13809 \\
             $300-600$ &         542 \\
            $600-1800$ &         298 \\
            $1800-3600$ &          62 \\
            $3600-7200$ &          31 \\
           $\geq 7200$ &          5 \\
\bottomrule
\end{tabular}
    \caption{Frequency distribution table of the time needed to generate unavoidable set}
    \label{Tab:FrequencyDistribution}
\end{table}

Figure~\ref{Fi:CutGenerationTime} does not catch how extreme these peaks can be. To see this effect, we provide the frequency distribution table of the generation time of unavoidable sets in  Table~\ref{Tab:FrequencyDistribution}. Though the majority of the minimally unavoidable sets ($94.10\%$) can be generated in less than 1 minute, some minimal unavoidable sets are very hard to find with the longest taking nearly 3 hours to find.

Lastly, it is important to remember that we are not obliged to generate all unavoidable sets since their sole function is to help reduce the feasible region of our bilevel program and improve performance. For this reason, we find it helpful to plot the average number of cuts generated as a function of time. To do this, for each instance $I$ and each $n \in [5000]$, we calculate the cumulative time our model takes to generate $n$ unavoidable sets of instance $I$. We then take the geometric mean of the cumulative time for each $n$ over all the instances and plot the result as a function of $n$. The resulting plot is shown in Figure~\ref{Fi:YieldCurve}. The figure reiterates that generating unavoidable sets is quicker in the beginning and shows how it becomes more difficult over time. It takes less than $20000$ seconds to generate the first $2000$ unavoidable sets and nearly $40000$ seconds to generate the next $2000$.

We will now test the effect of unavoidable set inequalities on our model by varying the number of inequalities that are used. For our initial analysis, we do not take into account the time needed to generate the unavoidable sets. We decide to test 500, 1000, 3000, and 5000 unavoidable set inequalities, where we use the first $n$ inequalities generated by our unavoidable set generating algorithm. A summary of the optimization results is shown in Table~\ref{Tab:SummaryDifferentCuts}. 
$45$ out of the $50$ instances of size $9 \times 9$ solved to optimality in at least one solver setting. Interestingly, all instances in the $17$ clues puzzle group solve to optimality in at least one solver setting and they generally solve faster than the instance group with no $17$ clue puzzle, see Figure~\ref{Fi:BoxPlot17Clue}.

\begin{table}[h]
    \centering
    \begin{tabular}{r|rr}
\toprule
 &\multicolumn{2}{c}{\# of instances}\\
 \# of unavoidable set &optimal &time limit \\
\midrule
500 &37 &13 \\
1000 &43 &7 \\
3000 &36 &14 \\
5000 &29 &21\\
\bottomrule
\end{tabular}
    \caption{Summary of end result for $9 \times 9$ standard bilevel model with different number of unavoidable set cuts}
    \label{Tab:SummaryDifferentCuts}
\end{table}
We plot the resulting performance profile \cite{eD02} in Figure~\ref{Fi:PerformanceProfile9by9}. We observe that adding too few or too many inequalities results in slower optimization times. Nearly $60\%$ of the instances solve fastest on models that use $1000$ unavoidable set inequalities, followed by slightly under $20\%$ of instances that solve fastest on models that use $500$ unavoidable set inequalities. This claim is also supported when we see that we solve to optimality in most instances when we are using $1000$ unavoidable sets inequalities. By looking deeper into node-level data as presented in Table~\ref{Tab:AverageNodeCount}, we see that too few inequalities result in a huge increase of nodes processed to prove optimality, while too many inequalities result in a huge decrease in node throughput. The best choice is therefore likely to be in the middle.
\begin{table}[h]
    \centering
    \begin{tabular}{r|rrr}
\toprule
 \# of unavoidable set &node count &time per node [s] & total runtime [s]\\
\midrule
500       & 2503382 &          0.028 &  70312 \\
1000      & 2097150 &          0.031 &  64941 \\
3000      & 1126612 &          0.082 &  92931 \\
5000      &  887962 &          0.139 & 123472 \\
\bottomrule
\end{tabular}
\caption{Geometric average of node count, time per node, and runtime of different settings}
\label{Tab:AverageNodeCount}
\end{table}

\begin{figure}[]
    \begin{center}
        \begin{subfigure}[b]{.49\linewidth}
        \begin{center}
\includegraphics[width=\textwidth]{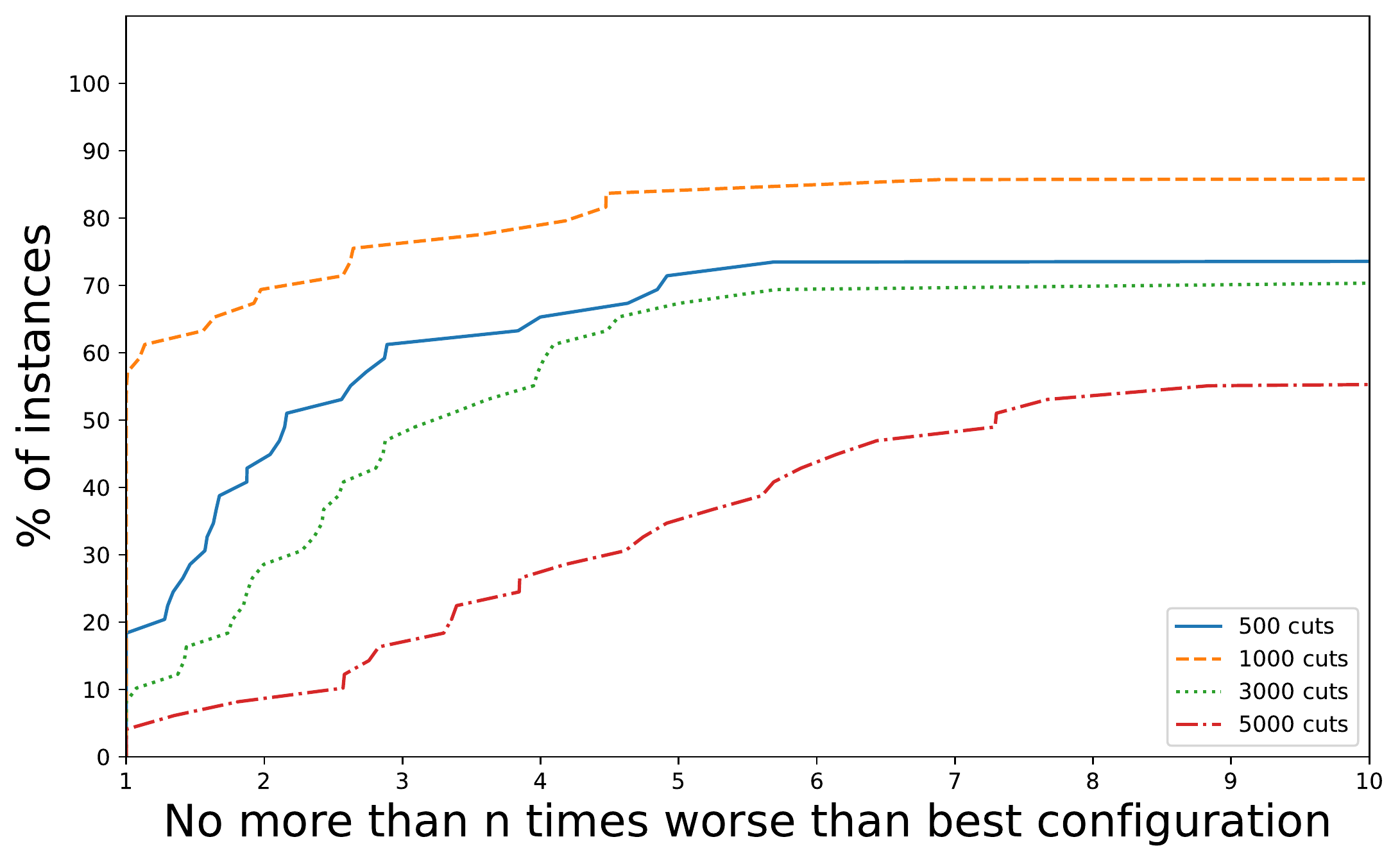}
        \caption{Performance profile of different number of unavoidable set on $9 \times 9$ instances}
        \label{Fi:PerformanceProfile9by9}

        \end{center}
        \end{subfigure}
        \begin{subfigure}[b]{.49\linewidth}
        \begin{center}
        \includegraphics[width=\textwidth]{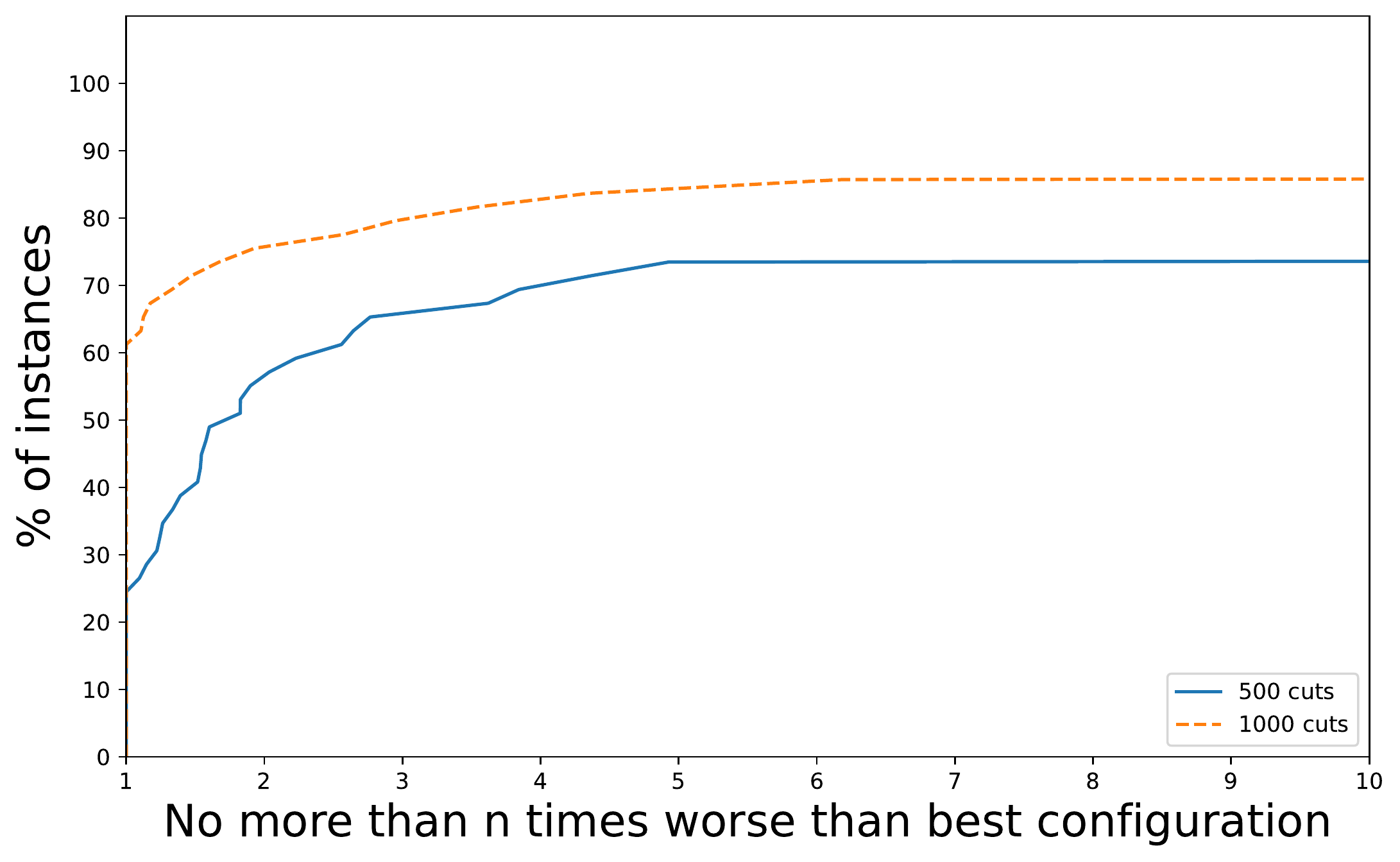}
\caption{Performance profile of different number of unavoidable set on $9 \times 9$ instances with augmented time}
\label{Fi:PerformanceProfile9by9Augmented}
        \end{center}
        \end{subfigure}
    \end{center}
    \caption{Experiment Results For Solving 9 by 9 instances}
\end{figure}
\begin{figure}
\begin{center}
        \includegraphics[width=0.4\textwidth]{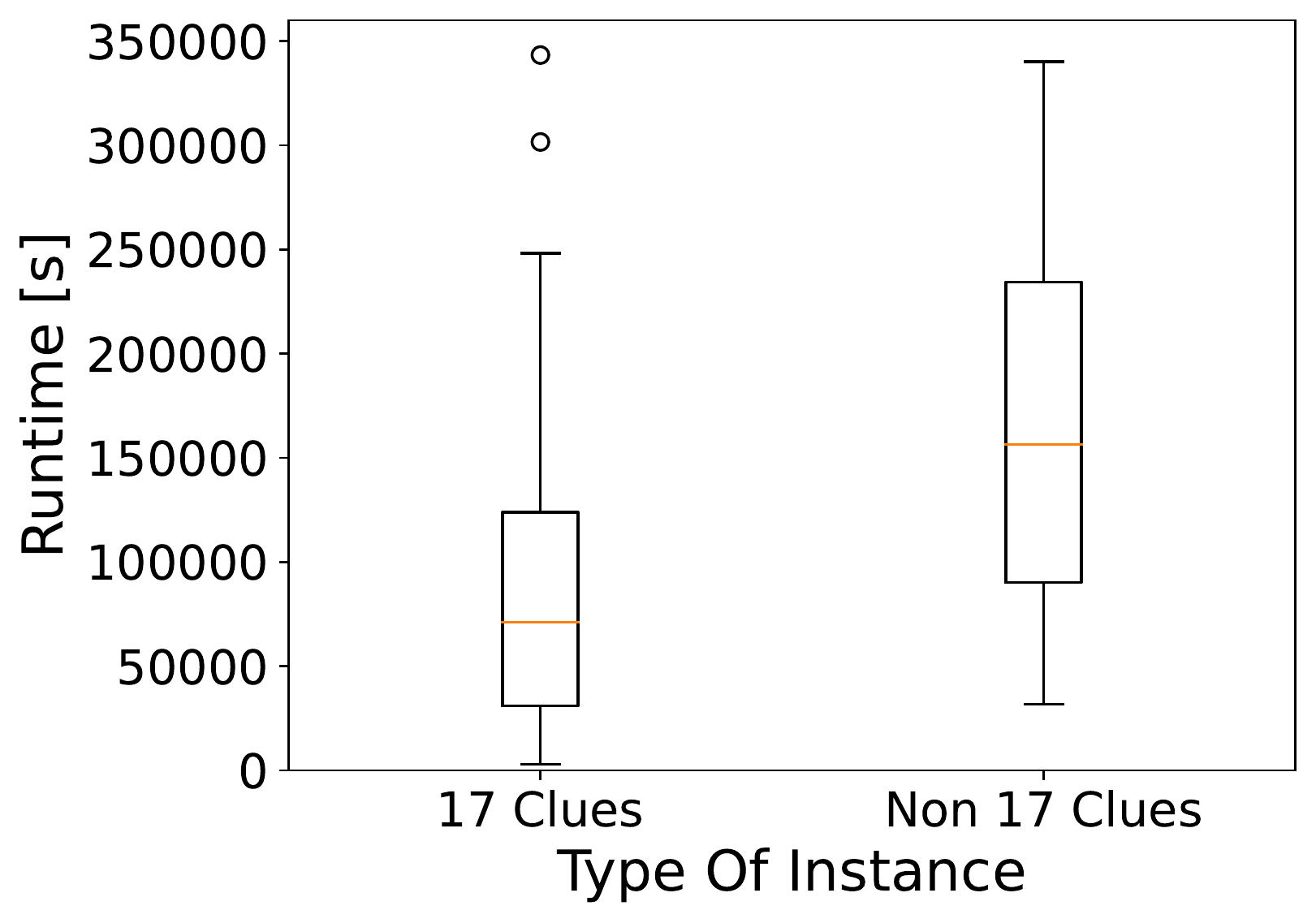}
\end{center}
\caption{Comparison Of Solving Time For 17 Clue Instances and Non 17 Clue Instances}
\label{Fi:BoxPlot17Clue}
\end{figure}
Finally, we take into account the time needed to generate the inequalities. Note that we only need to compare models with $500$ and $1000$ inequalities since $1000$ inequalities models outperform the $3000$ and $5000$ inequalities model cut models. To compare if it is worth generating the $500$ extra unavoidable sets, we compute the time needed to solve the bilevel instance plus the time needed to generate the unavoidable sets and plot the performance profile for instances with $500$ and $1000$ inequalities. The calculation is done by using the data from the generating unavoidable set experiments. The resulting plot is shown in \ref{Fi:PerformanceProfile9by9Augmented}. We see that even accounting for the unavoidable set cut generation time, using $1000$ inequalities is still superior to using $500$ inequalities.

For our experiments, we also obtained preliminary results with the \texttt{MiBS} solver~\cite{tS20}. Even on easy instances however, we quickly observed that \texttt{MiBS} required much more time than the solver from \cite{mF16}. We believe that this is in large part due to the inability of \texttt{MiBS} to find a primal solution to our problem. 

\section{Generalization of the Model to other Fewest Clue Problems}

A desire for unique solutions is not only relevant to Sudoku, with other example problems being Slither Link and Cross Sum~\cite{tY03}. This motivates the definition of the ``Fewest Clue Problem" (FCP) class in \cite{eD18}. In this section, we show how our model can also be adapted for FCP problems of other puzzles which have a linear binary formulation.

We restate the definition of the Fewest Clue Problem as in \cite{eD18}. Let $A$ be a problem in NP. We denote with $R_A$ the set of instance-certificate pairs where the certificates are binary strings of length $l$. For a given instance $I$ of $A$, we call a string $c \in \{0,1,\bot\}$ a \emph{clue} if there exists a certificate $c^*$ such that $(I,c^*) \in R_A$ and $c_i = c^*_i$ for all indices $i \in [k]$ where $c_i \neq \bot$. The symbol $\bot$ can be interpreted as a missing or non-specified entry.  We call $c^*$ a \emph{satisfying solution} to clue $c$. The \emph{size} of a clue is the number of non $\bot$ characters. 

We define $\mathrm{FCP}~A$ to be the decision problem: given an instance $I$, a certificate $c^*$ and an integer $k$, does there exist a clue $c$ of size at most $k$ for which the unique satisfying solution is $c^*$? We note that our definition is a slight variant of that proposed in \cite{eD18}.

We make the assumption that there exists an $l$-dimensional polytope $\mathcal{Q}$ such that $c$ is a valid certificate if and only if $c \in \mathcal{Q}$ and is binary.
The $\mathrm{FCP}~A$ can be written as a bilevel optimization problem as follows:
\begin{align*}
     \min_{x,y,z} \quad&\sum_{i=1}^l y_{i}\\
    \text{s.t.} \quad&z = 1 \\
    &y_{i} \in \{0,1\},\quad\forall\ i \in [l]\\
    &(x,z)\in S(y) 
\end{align*}
where $S(y)$ is the set of optimum solutions to the $y$-parameterized follower problem:
\begingroup
\allowdisplaybreaks
\begin{align*}
     \min_{x,z}\quad&z\\
    \text{s.t.}\quad& x\in \mathcal{Q}\\
    &x_{i} \geq y_{i}, \quad \forall\ i \in [l],\ c^*_i =1\\
    & \sum_{i\in [l], c_i^* = 1} x_{i} + \sum_{i\in [l], c_i^* = 0} (1-x_{i}) - z \leq l-1 \tag{$NG$}\\
    &x_{i},z \in \{0,1\}, \quad\forall\ i \in [l].
\end{align*}
\endgroup
The leader program determines which indices are given in the clue, while the follower tries to find an alternative solution respecting the clue.
Constraint~($NG$) is a no-good constraint prohibiting the assignment $x = c^*$ if $z = 0$.
It is trivially fulfilled if $z = 1$, it is a generalization of the equivalent constraint of the Sudoku-specific model presented in Section~\ref{sec:bilevel}.

\section{Conclusion and Outlook}

In this paper, we have shown that the Minimum Sudoku Clue problem can be formulated and solved as a binary bilevel linear programming problem. By introducing unavoidable-set inequalities, we showed that the formulation can be tightened, and that solver performance can be improved. Our models are able to compute a provable optimal solution to the Minimum Sudoku Clue problem in $95\%$ of instances. Despite these performance results, the inherent complexity of the Minimum Sudoku Clue problem and the more general Fewest Clue problem complicates scaling to larger instances. Unlike specialized ad hoc enumeration techniques developed in the Sudoku literature~\cite{d23} however, our approach naturally benefits from the continued improved performance of mixed-integer programming solvers. 

We see three main avenues of future research for the Minimum Sudoku Clue problem. First, we can use faster unavoidable set finding algorithms such as the one proposed by \cite{gM14}. Second, we can develop formulations that exploit the symmetries of Sudoku grids. Third, we can develop a branch-and-cut approach leveraging unavoidable set inequalities to separate non-feasible solutions throughout the branch-and-bound process instead of initially applying a large number of inequalities.

\section*{Acknowledgments}
The work for this article has been conducted in the Research Campus MODAL funded by the German Federal Ministry
of Education and Research (BMBF) (fund numbers 05M14ZAM, 05M20ZBM). The described research activities are
funded by the Federal Ministry for Economic Affairs and Energy within the project UNSEEN (ID: 03EI1004-C). We thank Markus Sinnl and coauthors for providing us a license to their bilevel solver, Fakultät II at the Technische Universität Berlin for allowing us to use their HPC facility, and Kai Hoppmann for initial advice on the integer formulation.

\bibliographystyle{siam}
\bibliography{bibliography.bib}

\end{document}